\documentclass[12pt]{amsart}

\usepackage{amssymb}

\usepackage{enumerate}

\usepackage{graphicx}

\usepackage{amsmath}
\usepackage{amsfonts}
\usepackage{amsmath}
\usepackage{amsthm}
\usepackage{amssymb}
\usepackage{amscd}
\usepackage{mathrsfs}
\usepackage{latexsym}
\usepackage{times}
\usepackage{fancyhdr}
\usepackage{epsfig}
\usepackage[all]{xy}
\usepackage{xypic}
\usepackage[english]{babel}
\usepackage{tkz-graph}

\usepackage{times}
\usepackage{fancyhdr}
\usepackage{epsfig}
\usepackage[all]{xy}
\usepackage[english]{babel}
\usepackage{leftidx}
\usepackage{enumerate}
\usepackage{graphicx}
\usepackage{epigraph}

\usepackage{xypic}
\usepackage[english]{babel}
\usepackage{tkz-graph}
\usepackage{longtable}
\usepackage{supertabular}

\usepackage{helvet}
\usepackage{textcomp}

\usepackage{float}
\usepackage{longtable,booktabs}

\usepackage{pifont}
\usepackage{tikz}

\usepackage{tikz}
\usepackage{etoolbox}
\usepackage{enumitem}
\usepackage{enumitem}
\usepackage{color}

\usepackage{calc}
\usepackage{tikz}
\usetikzlibrary{decorations.markings}
\tikzstyle{vertex}=[circle, draw, inner sep=0pt, minimum size=6pt]

\makeatletter
\@namedef{subjclassname@2010}{%
  \textup{2010} Mathematics Subject Classification}
\makeatother

\newtheorem{thm}{Theorem}[section]

\newtheorem{lem}[thm]{Lemma}
\newtheorem{prop}[thm]{Proposition}

\theoremstyle{definition}

\numberwithin{equation}{section}

\def\Gon{\text{Gon}}

\newcommand{\hooklongrightarrow}{\lhook\joinrel\longrightarrow}

\begin{document}


\baselineskip=17pt



\title[Small cyclic torsion]{On the small cyclic torsion of elliptic curves over cubic number fields.}

\author[J. Wang]{Jian Wang}

\address{College of Mathematics\\ Jilin Normal University\\
Siping, Jilin 136000, China}
\email{blandye@gmail.com}

\date{\today}

\begin{abstract}
Merel's result on the strong uniform boundedness conjecture made it meaningful to classify the torsion part of the Mordell-Weil groups of all elliptic curves defined over number fields of fixed degree $d$. In this paper, we discuss the cyclic torsion subgroup of elliptic curves over cubic number fields. For $N=49,40,25$ or $22$, we show that $\mathbb{Z}/N\mathbb{Z}$ is not a subgroup of $E(K)_{tor}$ for any elliptic curve $E$ over a cubic number field $K$.
\end{abstract}



\maketitle

\section{Introduction}

In 1996, Merel \cite{Merel} finally proved the strong uniform boundedness conjecture for elliptic curves over number fields.

\begin{thm}[Merel]
For every positive integer $d$, there exists an integer $B_d$ such that for every number field $K$ of degree $d$ and every elliptic curve $E$ over $K$, we always have $$|E(K)_{tor}|\leq B_d$$
\end{thm}

Merel's result made it meaningful to classify the torsion part of the Mordell-Weil groups of all elliptic curves defined over number fields of fixed degree $d$. The case $d=1$ was solved by Mazur \cite{Mazur} and Kubert \cite{Kubert}. The case $d=2$ was solved by Kamienny \cite{Kamienny1}, Kenku and Momose \cite{KenkuMomose}.

In \cite{Wang}, we proposed the aim of restricting the size of cyclic torsion subgroup of elliptic curves over cubic number fields. In that paper, six out of twenty four composite integers were ruled out with the help of the Kamienny's criterion. In this paper, we discuss several cases in the list of 24 which are too small to apply Kamienny's criterion.

The main result of this paper is the following:
\begin{thm}\label{T1} If $N=49,40,25$ or $22$, then $\mathbb{Z}/N\mathbb{Z}$ is not a subgroup of $E(K)_{tor}$ for any elliptic curve $E$ over a cubic number field $K$.
\end{thm}

\section{Preliminaries}

In this section, we omit the background materials which were covered in section 2 of \cite{Wang}. Readers who are interested may switch there.

Let $N$ be a positive integer. Let $X_1(N)$ (resp. $X_0(N)$) be the modular curve defined over $\mathbb{Q}$ associated to the congruence subgroup $\Gamma_1(N)$ (resp. $\Gamma_0(N)$). We denote by $Y_1(N)=X_1(N)\backslash\{cusps\}$,
$Y_0(N)=X_0(N)\backslash\{cusps\}$ the corresponding affine curves. Denote by $J_1(N)$ (resp. $J_0(N)$) the jacobian of $X_1(N)$ (resp. $X_0(N)$).

For a modular curve $X$, let $X^{(d)}$ be the $d$-th symmetric power of $X$. Define $$\Phi:X^{(d)}\longrightarrow J_X$$
by $\Phi(P_1+\cdots+P_d)=[P_1+\cdots+P_d-d\infty]$ where $J_X$ is the jacobian of $X$, and $[~~]$ denotes the divisor class. Let $Gon(X)$ denote the gonality of $X$. The following generalization of proposition 1(i) in Frey \cite{Frey}, which was proved in \cite{Wang}, is also necessary in section \ref{method}.

\begin{lem}[Frey]\label{Frey} Assume that $Gon(X)>d$ and $K$ is a finite extension of $\mathbb{Q}$. Then $\Phi_{|X^{(d)}(K)}$ is injective.
\end{lem}

In this paper, we are interested in the gonality of the modular curves $X_1(N)$.  Since the 1-gonal curves are precisely the curves of genus 0, then $X_1(N)$ is 1-gonal if and only if $N$ is among the eleven values $N=1-10,12$ with genus 0. The complete lists of 2-gonal and 3-gonal ones were determined by Ishii-Momose \cite{IshiiMomose} and Jeon-Kim-Schweizer \cite{JeonKimSchweizer}.

\begin{prop}[Ishii-Momose]\label{IshiiMomose} The modular curve $X_1(N)$ is 2-gonal if and only if $N$ is one of the following:
$$\aligned N&=1-10,12   &(g=0);\\
N&=11,14,15   &(g=1);\\
N&=13,16,18   &(g=2).\endaligned$$
\end{prop}

\begin{prop}[Jeon-Kim-Schweizer]\label{JeonKimSchweizer} The modular curve $X_1(N)$ is 3-gonal if and only if $N$ is one of the following:
$$\aligned N&=1-10,12   &(g=0);\\
N&=11,14,15   &(g=1);\\
N&=13,16,18   &(g=2);\\
N&=20   &(g=3).\endaligned$$
\end{prop}

Any noncuspidal point of $X_1(N)$ is represented by $(E,\pm P)$, where $E$ is an elliptic curve and $P\in E$ is a point of order $N$. Any noncuspidal point of $X_0(N)$ is represented by $(E, C)$, where $E$ is an elliptic curve and $C\subset E$ is a cyclic subgroup of order $N$. The map $\pi: X_1(N)\longrightarrow X_0(N)$ sends $(E,\pm P)$ to $(E,\langle P\rangle)$, where $\langle P\rangle$ is the cyclic subgroup generated by $P$.

Let $p$ be a prime such that $p\nmid N$. Igusa's theorem \cite{Igusa} says that the modular curves $X_1(N)$ and $X_0(N)$ have good reduction at prime $p$. The following theorem of Serre and Milne says that reducing the modular curve is compatible with reducing the modular interpretation.

\begin{thm}[Serre-Milne]\cite[Theorem 1]{Ogg75} Any point of $Y_1(N)$ or $Y_0(N)$, rational over a field $K$ (of characteristic not dividing $N$), is represented by a $K$-rational pair (i.e. $E$ is defined over $K$, and $P$ is rational over $K$, or $C$ is a group rational over $K$), and conversely.
\end{thm}

Let $K$ be a number field with ring of integers $\mathcal{O}_K$, $\wp\subset\mathcal{O}_K$ a prime ideal lying above $p$, $k=\mathbb{F}_q=\mathcal{O}_K/\wp$ its residue field. Let $E$ be an elliptic curve over $K$ and $P\in E(K)$ a point of order $N$. Let $\widetilde{E}$ be the fibre over $k$ of the N\'{e}ron model of $E$, and let $\widetilde{P}\in\widetilde{E}(k)$ be the reduction of $P$. The following well-known but rarely mentioned theorem \cite[Proposition 2.5]{Wang} shows that $\widetilde{P}$ has order $N$ when $p\nmid N$.

\begin{thm}\label{torsionreduction} Let $m$ be a positive integer relatively prime to $char(k)$. Then the reduction map
$$E(K)[m]\longrightarrow\widetilde{E}(k)$$
is injective.
\end{thm}

Let $k=\mathbb{F}_q$ be the finite field with $q=p^n$ elements. Let $E/k$ be an elliptic curve over $k$. Let $|E(k)|$ be the number of points of $E$ over $k$. Then Hasse's theorem states that
$$||E(k)|-q-1|\leq2\sqrt{q}$$
i.e.
$$(1-\sqrt{p^n})^2\leq|E(k)|\leq(1+\sqrt{p^n})^2$$
Let $t=q+1-|E(k)|$, $E$ is called ordinary if $(t,q)=1$, otherwise it is called supersingular. In the range proposed by Hasse's theorem, all the ordinary $t$ appear, while the supersingular $t$ only appears in restricted case. This was determined by Waterhouse \cite[Theorem 4.1]{Waterhouse}:

\begin{prop}[Waterhouse] \label{Waterhouse} The isogeny classes of elliptic curves over $k$ are in one-to-one correspondence with the rational integers $t$ having $|t|\leq2\sqrt{q}$ and satisfying one of the following conditions:
\newline\indent (1) $(t,p)=1$;
\newline\indent (2) If $n$ is even: $t=\pm2\sqrt{q}$;
\newline\indent (3) If $n$ is even and $p\not\equiv1\mod3$: $t=\pm\sqrt{q}$;
\newline\indent (4) If $n$ is odd and $p=2$ or $3$: $t=\pm p^{\frac{n+1}{2}}$;
\newline\indent (5) If either (i) $n$ is odd or (ii) $n$ is even and $p\not\equiv1\mod 4$: $t=0$.
\end{prop}


\section{Method}
\label{method}

The following Theorem states that the jacobian $J_1(N)$ decomposes to a direct sum of modular abelian varieties.

\begin{thm}\cite[Theorem 6.6.6]{DiamondShurman}\label{decomposition} The jacobian $J_1(N)$ is isogenous to a direct sum of abelian varieties (over $\mathbb{Q}$) associated to equavalence classes of newforms
$$J_1(N)\longrightarrow\bigoplus_fA_f^{m_f}$$
with $f(\tau)=\sum_{n=1}^\infty a_n(f)e^{2\pi in\tau}$ newforms of divisor level.
\end{thm}

The $L$-series $L(A_f,s)$ of $A_f$ coincides, up to a finite number of Euler factors, with the product
$$\prod_\sigma L(f^\sigma,s)=\prod_\sigma\sum_{n=1}^\infty a_n^\sigma n^{-s}$$
where $\sigma$ runs through embeddings $\sigma:K_f\hooklongrightarrow\mathbb{C}$ with $K_f=\mathbb{Q}(\{a_n\})$ the number field of $f$ (See \cite[\S 7.5]{Shimura}). The following proposition is a special case of Corollary 14.3 in Kato \cite{Kato}:

\begin{prop} Let $A$ be an abelian variety over $\mathbb{Q}$ such that there is a surjective homomorphism $J_1(N)\longrightarrow A$ for some $N\geq1$. If $L(A,1)\neq0$, then $A(\mathbb{Q})$ is finite.
\end{prop}

The decomposition of $J_1(N)$ and the non-vanishing of the $L$-series at $s=1$ of modular abelian varieties can be calculated in Magma \cite{Magma}. If $L(A_f,1)\neq0$ for all $A_f$, then we know $A_f(\mathbb{Q})$ is finite for all $A_f$, therefore $J_1(N)_{/\mathbb{Q}}$ is finite. For the $N\leq65$ in the list in \cite{Wang}, Table \ref{jacobian} is the result of calculations in Magma. The second column $t$ is the number of non-isogenous modular abelian varieties in the decompositon $J_1(N)=\bigoplus_{i=1}^tA_i^{m_i}$. The third column list the dimension $d_i$ and multiplicity $m_i$ of each $A_i$ (we omit $m_i$ if $m_i=1$). The fourth column verifies vanishing of $L$-series at $1$ ($T$ means $L(A_i,1)=0$ and $F$ means $L(A_i,1)\neq0$). It is easy to see $J_1(N)(\mathbb{Q})$ is finite for the $N$'s in Table \ref{jacobian} except $N=65,63$.

\begin{table}[!ht]
\tabcolsep 0pt
\vspace*{0pt}
\begin{center}
\def\temptablewidth{0.8\textwidth}
\setlength{\abovecaptionskip}{0pt}
\setlength{\belowcaptionskip}{-5pt}
\caption{Decompostion of $J_1(N)$}
\label{jacobian}
{\rule{\temptablewidth}{1pt}}
\begin{tabular*}{\temptablewidth}{@{\extracolsep{\fill}}ccccccccccccc}
~~~~$N$~~~~&$t$&$d_i(m_i)$&$L(A_i,1)=0$\\\hline
$49$&$5$&$1,48,6,12,2$&$F,F,F,F,F$\\
$25$&$2$&$8,4$&$F,F$\\
$27$&$2$&$1,12$&$F,F$\\
$32$&$4$&$1,4,8,2(2)$&$F,F,F,F$\\
$65$&$19$&$1,2,2,6,20,20,8,2(2),8,$&$T,F,F,F,F,F,F,F,F,$\\
&&$2,8,2,8,4,4,12,6,2,2$&$F,F,F,F,F,F,F,F,F,F$\\
$39$&$10$&$1,2,4,8,2,2(2),4,2,4,2$&$F,F,F,F,F,F,F,F,F,F$\\
$26$&$5$&$1,1,2(2),2,2$&$F,F,F,F,F$\\
$55$&$10$&$1,2,1(2),4,32,8,8,16,4,4$&$F,F,F,F,F,F,F,F,F,F$\\
$33$&$6$&$1,1(2),8,4,4,2$&$F,F,F,F,F,F$\\
$22$&$2$&$1(2),4$&$F,F$\\
$35$&$8$&$1,2,2,4,4,4,4,4$&$F,F,F,F,F,F,F,F$\\
$63$&$20$&$1,2,1(2),6,6,2,10,4,2(2),2,$&$F,F,F,F,F,F,F,F,F,T$\\
&&$10,2,2,2(2),2,10,2,10,12,4$&$F,F,F,F,F,F,F,F,F,F$\\
$28$&$4$&$1(2),4,2,2$&$F,F,F,F$\\
$45$&$8$&$1,1(2),2,6,16,4,2,8$&$F,F,F,F,F,F,F,F$\\
$30$&$4$&$1,1(2),4,2$&$F,F,F,F$\\
$40$&$7$&$1,1(2),4,2(2),8,2,4$&$F,F,F,F,F,F,F$\\
$36$&$5$&$1,8,2,2(2),2$&$F,F,F,F,F$\\
$24$&$3$&$1,2,2$&$F,F,F$\\

\end{tabular*}
{\rule{\temptablewidth}{1pt}}
\end{center}
\end{table}

In the proof of Lemma \ref{reductionlemma}, we use a specialization lemma in Appendix of Katz \cite{Katz} and a theorem of Manin \cite{Manin} and Drinfeld \cite{Drinfeld}.

\begin{lem}[Specialization Lemma]\label{Katz} Let $K$ be a number field. Let $\wp\subset\mathcal{O}_K$ be a prime above $p$. Let $A/K$ be an abelian variety. Suppose the ramification index $e_\wp(K/\mathbb{Q})<p-1$. Then the reduction map
$$\Psi: A(K)_{tor}\longrightarrow A(\overline{\mathbb{F}}_p)$$
is injective.
\end{lem}

\begin{thm}[Manin-Drinfeld]\label{ManinDrinfeld} Let $C\subset SL_2(\mathbb{Z})/(\pm1)$ be a congruence subgroup. $x,y\in\mathbb{P}^1(\mathbb{Q})$ and $\overline{x},\overline{y}$ the images of $x$ and $y$ respectively, on $\overline{\mathbb{H}}/C$. Then the class of divisors $(\overline{x})-(\overline{y})$ on curve $\overline{\mathbb{H}}/C$ has finite order.
\end{thm}

\begin{lem} \label{reductionlemma}
Suppose$N>4$ such that  $\Gon(X_1(N))>d$, $J_1(N)(\mathbb{Q})$ is finite, $p>2$ is a prime not dividing $N$. Let $K$ be a number field of degree $d$ over $\mathbb{Q}$ and $\wp$ a prime of $K$ over $p$. Let $E_{/K}$ be an elliptic curve with a $K$-rational point $P$ of order $N$, i.e. $x=(E,\pm P)\in Y_1(N)(K)$. Then $E$ has good reduction at $\wp$.
\end{lem}

\begin{proof} Suppose $E$ has additive reduction at $\wp$, then $\widetilde{E}(k)^0\cong\mathbb{G}_{a/k}$ with $|\mathbb{G}_{a/k}|=p^i, i\leq d$ and $\widetilde{E}(k)/\widetilde{E}(k)^0\cong G$ with $|G|\leq4$.
Since $\widetilde{P}$ is a $k$-rational point of order $N$ in $\widetilde{E}$, then $N$ divides $|\widetilde{E}(k)|=|\mathbb{G}_{a/k}||G|$, which is impossible under our assumption.

Suppose $E$ has multiplicative reduction at $\wp$, i.e. $x$ specializes to a cusp of $\widetilde{X}_1(N)$. Then $\tau_i(K)$ is also a cubic field with prime ideal $\tau_i(\wp)$ over $p$ and residue field $k_i=k$. And $\tau_i(E)$ also has multiplicative reduction at $\tau_i(\wp)$. This means all the images $x_1,\cdots, x_d$ of $x$ specialize to cusps of $\widetilde{X}_1(N)$. Let $c_1,\cdots, c_d$ be the cusps such that
$$x_i\otimes\overline{\mathbb{F}}_p=c_i\otimes\overline{\mathbb{F}}_p, ~~~~~1\leq i\leq d$$

We know all the cusps of $X_1(N)$ are defined over $\mathbb{Q}(\zeta_N)$ \cite{Ogg73}. Let $\wp'$ be a prime in $\mathbb{Q}(\zeta_N)$ over $p$. It is an elementary fact in algebraic number theory that $p$ ramifies in $\mathbb{Q}(\zeta_N)$ if and only if $p|N$, so $e_{\wp'}(\mathbb{Q}(\zeta_N)/\mathbb{Q})=1$ under our assumption $p\nmid N$. So by Lemma \ref{Katz}, the specialization map
$$\Psi: J_1(N)(\mathbb{Q}(\zeta_N))_{tor}\longrightarrow J_1(N)(\overline{\mathbb{F}}_p)$$
is injective.

Since $\Gon(X_1(N))>d$, then by Lemma \ref{Frey}, the map
$$\Phi: X_1(N)^{(d)}(\mathbb{Q}(\zeta_N))\longrightarrow J_1(N)(\mathbb{Q}(\zeta_N))$$
is injective.

We know $x_1+\cdots+x_d$ is $\mathbb{Q}$-rational and $J_1(N)(\mathbb{Q})$ is finite, so $[x_1+\cdots+x_d-d\infty]$  is in $J_1(N)(\mathbb{Q}(\zeta_N))_{tor}$. By Theorem \ref{ManinDrinfeld}, the difference of two cusps of $X_1(N)$ has finite order in $J_1(N)$. So $[c_1+\cdots+c_d-d\infty]$ is also in $J_1(N)(\mathbb{Q}(\zeta_N))_{tor}$. Therefore $\Psi\circ\Phi(x_1+\cdots+x_d)=\Psi\circ\Phi(c_1+\cdots+c_d)$ implies $x_1+\cdots+x_d=c_1+\cdots+c_d$ since $\Psi\circ\Phi$ is injective. This is a contradiction because we assume $x$ is a noncuspidal point.

Therefore $E$ has good reduction at $\wp$.
\end{proof}

\section{Proof of Theorem \ref{T1}}\label{proof}

\subsection{$N=49,40$}
As is seen in Table \ref{jacobian}, $J_1(N)(\mathbb{Q})$ is finite. By Proposition \ref{IshiiMomose} and \ref{JeonKimSchweizer}, we know $Gon(X_1(N))>3$. Let $K$ be a cubic field and $\wp$ a prime of $K$ over $3$. Suppose $x=(E,\pm P)\in Y_1(N)(K)$. Therefore by Lemma \ref{reductionlemma}, $E$ has good reduction at $\wp$. By Theorem \ref{torsionreduction}, the reduction $\widetilde{P}$ of $P$ is a $k$-rational point of order $N$ in the elliptic curve $\widetilde{E}$ over $k=\mathcal{O}_K/\wp$.

But $\widetilde{E}(k)$ can not have a point of order $N$ since $N>(1+\sqrt{3^3})^2\approx38.4$. This is a contradiction. So $\mathbb{Z}/N\mathbb{Z}$ is not a subgroup of $E(K)_{tor}$.

\subsection{$N=25,22$}
As is seen in Table \ref{jacobian}, $J_1(N)(\mathbb{Q})$ is finite. By Proposition \ref{IshiiMomose} and \ref{JeonKimSchweizer}, we know $Gon(X_1(N))>3$. Let $K$ be a cubic field and $\wp$ a prime of $K$ over $3$. Suppose $x=(E,\pm P)\in Y_1(N)(K)$. Therefore by Lemma \ref{reductionlemma}, $E$ has good reduction at $\wp$. By Theorem \ref{torsionreduction}, the reduction $\widetilde{P}$ of $P$ is a $k$-rational point of order $N$ in the elliptic curve $\widetilde{E}$ over $k=\mathcal{O}_K/\wp$.

If $k=\mathbb{F}_3$ or $\mathbb{F}_{3^2}$, then $\widetilde{E}(k)$ can not have a point of order $N$ since $N>(1+\sqrt{3^2})^2$. If $k=\mathbb{F}_{3^3}$, suppose $\widetilde{E}(k)$ has a point of order $N$, then $\widetilde{E}(k)\cong\mathbb{Z}/N\mathbb{Z}$ since $Nm>(1+\sqrt{3^3})^2$ for any $m>1$. But by Theorem \ref{Waterhouse}, $|\widetilde{E}(k)|\neq N$ ($t=3$ for $N=25$, $t=6$ for $N=22$). This is a contradiction. So $\mathbb{Z}/N\mathbb{Z}$ is not a subgroup of $E(K)_{tor}$.

\section*{}
\subsection*{Acknowledgements}
We thank Sheldon Kamienny for providing many valuable ideas and insightful comments.


\begin{thebibliography}{HD}






\normalsize
\baselineskip=17pt















\bibitem{DiamondShurman} Diamond, F., Shurman, J.: \emph{A first course in modular forms.} Graduate Texts in Mathematics, 228. Springer-Verlag, New York (2005)

\bibitem{Drinfeld} Drinfeld, V. G.: \emph{Two theorems on modular curves.} (Russian) Funkcional. Anal. i Prilo\v{z}en.  7, no. 2, 83-84 (1973)




\bibitem{Frey} Frey, G.: \emph{Curves with infinitely many points of fixed degree.} Israel J. Math. 85, no. 1-3, 79-83 (1994)









\bibitem{Igusa} Igusa, J.: \emph{Kroneckerian model of fields of elliptic modular functions.} Amer. J. Math.  81,  561-577 (1959)


\bibitem{IshiiMomose} Ishii, N., Momose, F.: \emph{Hyperelliptic modular curves.} Tsukuba J. Math. 15, no. 2, 413-423 (1991)





\bibitem{JeonKimSchweizer} Jeon, D., Kim, C. H., Schweizer, A.: \emph{On the torsion of elliptic curves over cubic number fields.} Acta Arith.  113,  no. 3, 291-301 (2004)





\bibitem{Kamienny1} Kamienny, S.: \emph{Torsion points on elliptic curves and q-coefficients of modular forms.} Invent. Math.  109,  no. 2, 221-229 (1992)



\bibitem{Kato} Kato, K.: \emph{p-adic Hodge theory and values of zeta functions of modular forms.}
Ast\'{e}risque, (295):ix, 117-290. Cohomologies p-adiques
et applications arithm\'{e}tiques. III (2004)

\bibitem{Katz}  Katz, N. M.: \emph{Galois properties of torsion points on abelian varieties.} Invent. Math. 62 , no. 3, 481-502 (1981)




\bibitem{KenkuMomose} Kenku, M. A., Momose, F.: \emph{Torsion points on elliptic curves defined over quadratic fields.} Nagoya Math. J.  109, 125-149 (1988)












\bibitem{Kubert} Kubert, D. S.: \emph{Universal bounds on the torsion of elliptic curves.} Proc. London Math. Soc. (3) 33, no. 2, 193-237 (1976)





\bibitem{Magma} Magma: http://magma.maths.usyd.edu.au/magma/


\bibitem{Manin} Manin, Y. I.: \emph{Parabolic points and zeta functions of modular curves.} (Russian) Izv. Akad. Nauk SSSR Ser. Mat.  36, 19-66 (1972)


\bibitem{Mazur} Mazur, B.: \emph{ Modular curves and the Eisenstein ideal.} Inst. Hautes \'{E}tudes Sci. Publ. Math. No. 47 (1977), 33-186 (1978)






\bibitem{Merel} Merel, L.: \emph{Bornes pour la torsion des courbes elliptiques sur les corps de nombres.} (French) Invent. Math.  124,  no. 1-3, 437-449 (1996)









\bibitem{Ogg73} Ogg, A. P.: \emph{Rational points on certain elliptic modular curves.} Analytic number theory (Proc. Sympos. Pure Math., Vol XXIV, St. Louis Univ., St. Louis, Mo., 1972), pp. 221-231. Amer. Math. Soc., Providence, R.I. (1973)


\bibitem{Ogg75} Ogg, A. P.: \emph{Diophantine equations and modular forms.} Bull. Amer. Math. Soc. 81, 14-27 (1975)













\bibitem{Shimura} Shimura, G.: \emph{Introduction to the arithmetic theory of automorphic functions.} Kan\^{o} Memorial Lectures, No. 1. Publications of the Mathematical Society of Japan, No. 11. Iwanami Shoten, Publishers, Tokyo; Princeton University Press, Princeton, N.J. (1971)










\bibitem{Wang} Wang, J.: \emph{On the cyclic torsion of elliptic curves over cubic number fields.} J. Number Theory 183, 291-308 (2018)


\bibitem{Waterhouse} Waterhouse, W. C.: \emph{Abelian varieties over finite fields.} Ann. Sci. \'{E}cole Norm. Sup. (4)  2  521-560 (1969)



\end{thebibliography}
\end{document}